\documentclass[a4paper, 12pt]{amsart}
\usepackage{mathrsfs}
\usepackage{amsmath}
\usepackage{amssymb}
\usepackage{verbatim}
\usepackage{xcolor}
\usepackage{CJK}
\usepackage{graphicx}
\usepackage[all]{xy}
\usepackage{appendix}
\DeclareGraphicsExtensions{.eps,.ps,.jpg,.bmp}

\theoremstyle{plain}

\newtheorem{theorem}{Theorem}[section]
\newtheorem{proposition}{Proposition}[section]

\newtheorem{lemma}{Lemma}[section]

\newtheorem{problem}{Problem}[section]

\theoremstyle{definition}
\newtheorem{definition}{Definition}[section]

\theoremstyle{remark}
\newtheorem{remark}{Remark}[section]

\title{On the modified ideal sheaf}
\author{Jingcao Wu}

\begin{document}
\pagestyle{plain}
\begin{abstract}
It is a sequel to (Wu in arXiv:2003.05187). In that paper, we introduce a notion called modified ideal sheaf in order to make an asymptotic estimate for the order of the cohomology group. Here we continue to a general discussion about this notion. As an application, we study the direct images associated with a pseudo-effective line bundle.
\end{abstract}
\maketitle

\section{Introduction}
The multiplier ideal sheaf is a powerful tool in complex geometry. It is first introduced by Nadal in his work \cite{Nad90} concerning the existence of the K\"{a}hler--Einstein metric on a Fano manifold. Then numerous geometers use it to make great work. There are tremendous literatures, taking \cite{Cao14,DEL00,Dem92,Dem12,Mat14,Siu98} as examples.

One basic reason for these gorgeous applications is that it establishes the link between the analytic and algebraic geometry. In fact, in \cite{Nad90} the multiplier ideal sheaf is defined as the ideal sheaf whose germs are the holomorphic functions with $L^{2}$-bounded norm against a singular weight function. While in algebraic geometry, it is defined through its algebraic natures. One refers \cite{Laz04} for the algebraic description of the multiplier ideal sheaf. In our paper \cite{Wu20}, the multiplier ideal sheaf also plays an important role. However, in practice, we found that it is not always enough to ask the functions to be $L^{2}$-bounded.

More precisely, we were trying to estimate the $L^{\infty}$-norm of an $L$-valued $(n,q)$-form $\alpha$ against a singular metric $h$ via its $L^{2}$-norm in \cite{Wu20} (Proposition 3.2). Here $L$ is a line bundle.  Namely, it is desired to prove that
\begin{equation}\label{e11}
\sup|\alpha(x)|^{2}_{h}\leqslant C\int_{X}|\alpha|^{2}_{h}dV
\end{equation}
on a compact complex manifold $X$. The problem is that if we only assume the $L^{2}$-norm to be bounded, we cannot guarantee the $L^{\infty}$-norm is bounded a priori since $h$ may be infinite somewhere. So it seems a tough work to establish such an inequality.

As a compromise, we introduce the modified ideal sheaf $\mathfrak{I}(h)$. The basic idea is to ask the vanishing order of $f$ along the poles of $h$ a little bit higher than the case that $f\in\mathscr{I}(h)$. We collect those functions as the germs of $\mathfrak{I}(h)$. Here we use $\mathscr{I}(h)$ to present the traditional multiplier ideal sheaf. Obviously, $\mathfrak{I}(h)\subset\mathscr{I}(h)$. The explicit definition will be given in Sect.3.

We briefly recall it here for readers' convenience. When $h$ has analytic singularities, it can be computed precisely \cite{Dem12} with the help of a log-resolution $\mu:\tilde{X}\rightarrow X$. Indeed, let Jacobian divisor be $R=\sum\rho_{j}D_{j}$ where $D_{j}$ is the exceptional divisor, and let $g_{j}$ be the generator of $D_{j}$. Then the weight function $\varphi$ of $h$ after pull-back can be written as
\[
\varphi\circ\mu=a\sum\lambda_{j}\log|g_{j}|^{2}.
\]
In this situation,
\[
\mathscr{I}(h)=\mu_{\ast}\mathcal{O}_{\tilde{X}}(\sum(\rho_{j}-\lfloor a\lambda_{j}\rfloor)D_{j})
\]
with $\lfloor a\lambda_{j}\rfloor$ referring to the round down. Then its modified ideal sheaf is defined as
\[
\mathfrak{I}(h):=\mu_{\ast}\mathcal{O}_{\tilde{X}}(\sum(-\lceil a\lambda_{j}\rceil)D_{j})
\]
with $\lceil a\lambda_{j}\rceil$ referring to the round up. Then we use Demailly's approximation technique \cite{DPS01} to define the modified ideal sheaf for a metric not necessary with analytic singularities.

If we use the modified ideal sheaf instead of the original one, with loss of a little generality, we will get satisfying result. More specific, we can truly prove (\ref{e11}) (Theorem 1.1, \cite{Wu20}), hence an asymptotic estimate (Theorem 1.2, \cite{Wu20}) when $L$ is nef. This estimate is new and has good consequences, such as a partial answer to Demailly--P\u{a}un's conjecture (Theorem 1.3, \cite{Wu20}) and a Nadel-type vanishing theorem (Theorem 1.4, \cite{Wu20}). We believe that there are more treasures waiting to excavate in this notion. It's the motivation for this note.

To begin with, we list a few basic properties of the modified ideal sheaf, which more or less appeared in \cite{Wu20}.
\begin{theorem}\label{t1}
Let $X,X_{1},X_{2}$ be compact complex manifolds, and let $L,L_{1},L_{2}$ be pseudo-effective line bundles on $X$, $X_{1}$ and $X_{2}$ respectively. Assume that $\varphi,\varphi_{1},\varphi_{2}$ are singular metrics on $L$, and $\psi_{1},\psi_{2}$ are singular metrics on $L_{1},L_{2}$. Keep the notations. Then we have
\begin{enumerate}
  \item $\mathfrak{I}(\varphi)\subset\mathscr{I}(\varphi)$. There exits a quasi-plurisubharmonic function $\psi$ on $X$ such that
  \[
  \mathfrak{I}(\varphi)=\mathscr{I}(\psi).
  \]
  In particular, $\mathfrak{I}(\varphi)=\mathscr{I}(\varphi)$ iff $\varphi$ is defined by a normal crossing divisor.
  \item If $f\in\mathfrak{I}(\varphi)$, $|f|^{2}e^{-\varphi}$ vanishes near the poles unless $\varphi$ can be approximated by a family of metrics $\{\varphi_{\varepsilon}\}$ with algebraic singularities. In particular, it is always bounded.
  \item $\mathfrak{I}(\varphi)=\mathcal{O}_{X}$ iff $\varphi$ has no poles.
  \item (Superadditivity, I) Let $\pi_{i}:X_{1}\times X_{2}\rightarrow X_{i}$, $i=1,2$ be the projections. Then
  \[
  \mathfrak{I}(\psi_{1}\circ\pi_{1}+\psi_{2}\circ\pi_{2})=\pi^{\ast}_{1}\mathfrak{I}(\psi_{1})\cdot\pi^{\ast}_{2}\mathfrak{I}(\psi_{2}).
  \]
  \item (Superadditivity, II) $\mathfrak{I}(\varphi_{1})\cdot\mathfrak{I}(\varphi_{2})\subset\mathfrak{I}(\varphi_{1}+\varphi_{2})$.
\end{enumerate}
\end{theorem}

Then we develop the harmonic theory with respect to the modified ideal sheaf. Assume that $X$ is a compact K\"{a}hler manifold, and $(L,\varphi)$ is a pseudo-effective line bundle. We use Demailly's approximation technique \cite{DPS01}, to approximate $\varphi$ by a family of metrics $\{\varphi_{\varepsilon}\}$ such that $\varphi_{\varepsilon}$ is smooth outside a subvariety $Z$. Now we define the $\Box_{\varphi}$-harmonic form $\alpha$ as the $L^{2}$-limit of a sequence of $\Box_{\varphi_{\varepsilon}}$-harmonic forms $\{\alpha_{\varepsilon}\}$ on $Y=X-Z$. Moreover, we ask that $\alpha$ and $\alpha_{\varepsilon}$ differ by a $\bar{\partial}$-exact form on $Y$. The space of $\Box_{\varphi}$-harmonic $L$-valued $(p,q)$-form is denoted by
\[
\mathcal{H}^{p,q}(X,L,\Box_{\varphi}).
\]
The whole details has been given in \cite{Wu20}, and we will present them in the text for readers' convenience. The harmonic spaces associated with $\varphi$ are then defined as
\[
\mathcal{H}^{p,q}(X,L,\mathscr{I}(\varphi)):=\{\alpha\in\mathcal{H}^{p,q}(X,L,\Delta_{\varphi});\int_{X}|\alpha|^{2}e^{-\varphi}<\infty\},
\]
and
\[
\mathcal{H}^{p,q}(X,L,\mathfrak{I}(\varphi)):=\{\alpha\in\mathcal{H}^{p,q}(X,L,\Delta_{\varphi});\int_{X}|\alpha|^{2}e^{-\psi}<\infty\}.
\]

Notice that although we define $\alpha$ as a limit on $Y$, it actually has a good extension property. We actually have the following singular version of Hodge's theorem.
\begin{theorem}[A singular version of Hodge's theorem]\label{t2}
Let $(X,\omega)$ be a compact K\"{a}hler manifold, and let $(L,\varphi)$ be a pseudo-effective line bundle on $X$. Assume that there exists a section $s$ of some multiple $L^{k}$ such that $\sup_{X}|s|_{k\varphi}<\infty$. Then the following relationship hold:
\begin{equation}\label{e12}
\begin{split}
   &\mathcal{H}^{n,q}(X,L\otimes\mathscr{I}(\varphi),\Box_{\varphi})\simeq H^{n,q}(X,L\otimes\mathscr{I}(\varphi)),\\
   &\mathcal{H}^{n,q}(X,L\otimes\mathfrak{I}(\varphi),\Box_{\varphi})\\
   \simeq&\mathrm{Im}(i_{n,q}:H^{n,q}(X,L\otimes\mathfrak{I}(\varphi))\rightarrow H^{n,q}(X,L\otimes\mathscr{I}(\varphi))).
\end{split}
\end{equation}
In particular, when $\varphi$ is smooth, $\alpha\in\mathcal{H}^{n,q}_{\leqslant0}(X,L,\Box_{\varphi})$ if and only if $\alpha$ is $\Box_{\varphi}$-harmonic in the usual sense.
\end{theorem}

We remark here that Theorem \ref{t2} works for a compact complex manifold if $L$ is moreover assumed to be nef. This statement can be found in \cite{Wu20}. It is worth to mention that in \cite{Mat14}, a similar result (Lemma 3.2) has been shown for a line bundle $(L,h)$ such that $h$ is smooth outside a subvariety and $i\Theta_{L,h}\geqslant0$. Theorem \ref{t2}, which benefits a lot form their work, generalises it.

Based on this theorem, we can extend main results of J. Koll\'{a}r in \cite{Ko86a,Ko86b} to the case that involves the modified ideal sheaf. We surmise them as the following theorem. It includes a Koll\'{a}r-type injectivity theorem and a torsion freeness theorem. This type of generalisation has been fully studied in recent years by \cite{Fuj12,FuM16,GoM17,Mat14,Mat15,Mat16,Siu82}.
\begin{theorem}\label{t3}
Let $f:X\rightarrow Y$ be a smooth fibration between two compact K\"{a}hler manifolds, and let $(L,\varphi)$ be a pseudo-effective line bundle on $X$. Assume that $s$ is a section of some multiple $L^{k-1}$ such that
\[
s\in H^{0}(X,L^{k-1}\otimes\mathfrak{I}((k-1)\varphi)).
\]
Then the following theorems hold:

1. [Injectivity theorem] The following map
\[
\mathcal{H}^{n,q}(X,L\otimes\mathfrak{I}(\varphi))\xrightarrow{\otimes s}\mathcal{H}^{n,q}(X,L^{k}\otimes\mathfrak{I}(k\varphi))
\]
induced by tensor with $s$ is injective.

2. [Torsion freeness] If $\varphi|_{X_{y}}$ is well-defined (i.e. not identically equal to $\infty$) for all $y\in Y$, the image of
\[
i_{n,q}:R^{q}f_{\ast}(K_{X/Y}\otimes L\otimes\mathfrak{I}(\varphi))\rightarrow R^{q}f_{\ast}(K_{X/Y}\otimes L\otimes\mathscr{I}(\varphi))
\]
is reflexive for all $q\geqslant0$.
\end{theorem}


\section{Preliminary}
\subsection{Positivity of the line bundle}
Let $X$ be a compact complex manifold, and let $L$ be a line bundle on $X$. Fix a Hermitian metric $\omega$ on $X$. We list some basic notions concerning the positivity of $L$ as follows.
\begin{definition}\label{d21}
1. $L$ is called positive if there exits a smooth metric $\varphi$ on $L$ such that $i\Theta_{L,\varphi}>0$. Equivalently, $L$ is ample, i.e. the sections of some multiple $L^{k}$ induce a embedding from $X$ to a projective space.

2. $L$ is called nef (numerically effective) if there exits a family of smooth metrics $\{\varphi_{\varepsilon}\}$ on $L$ such that $i\Theta_{L,\varphi_{\varepsilon}}\geqslant-\varepsilon\omega$ for every $\varepsilon$.

3. $L$ is called big if there exits a singular metric $\varphi$ on $L$ such that $i\Theta_{L,\varphi}>0$ in the sense of current.

4. $L$ is called pseudo-effective if there exits a singular metric $\varphi$ on $L$ such that $i\Theta_{L,\varphi}\geqslant0$ in the sense of current.
\end{definition}

The multiplier ideal sheaf is a powerful tool when dealing with the singular metric.
\begin{definition}\label{d22}
Let $L$ be a line bundle and let $\varphi$ be a singular metric on $L$. The stalk of the associated multiplier ideal sheaf at one point $x\in X$ is defined as:
\[
   \mathscr{I}(\varphi)_{x}:=\{f\in\mathcal{O}_{X,x};|f|^{2}e^{-\varphi}\textrm{ is integrable around }x\}.
\]
\end{definition}

\subsection{Positivity of the vector bundle}
Let $X$ be a compact complex manifold, and let $E$ be a vector bundle of rank $r$ on $X$. Let
\[
\hat{X}:=\mathbb{P}(E^{\ast})
\]
be the projectived bundle, and let $\mathcal{O}_{E}(1):=\mathcal{O}_{\mathbb{P}(E^{\ast})}(1)$ be the tautological line bundle on $\hat{X}$. Moreover, $\pi:\hat{X}\rightarrow X$ is the projection. We extend those positivity notions on a line bundle to $E$ by concerning $\mathcal{O}_{E}(1)$.
\begin{definition}\label{d23}
1. $E$ is called ample if $\mathcal{O}_{E}(1)$ is ample.

2. $L$ is called nef if $\mathcal{O}_{E}(1)$ is nef.

3. $L$ is called big if $\mathcal{O}_{E}(1)$ is big.

4. $L$ is called pseudo-effective if $\mathcal{O}_{E}(1)$ is pseudo-effective, and the image of the non-nef locus $\textrm{NNef}(\mathcal{O}_{E}(1))$ under $\pi$ is proper. The non-nef locus refers to the union of the curves $C$ in $\hat{X}$ such that $\mathcal{O}_{E}(1)\cdot C<0$.
\end{definition}

In practise, we will not always deal with the vector bundle. So it is necessary to consider more general sheaves, such as torsion free sheaves and reflexive sheaves. Let $\mathcal{E}$ be a torsion free coherent sheaf. It is then locally free outside a $2$-codimensional subvariety $Z$ of $X$. Furthermore, the projectived sheaf $\mathbb{P}(\mathcal{E})$ and the tautological divisor $\mathcal{O}_{\mathcal{E}}(1)$ are well-defined in this situation \cite{Har77}. So Definition \ref{d23} extends to the torsion free coherent sheaf. In summary, $\mathcal{E}$ is called ample (resp. nef, big, pseudo-effective) if $\mathcal{E}|_{X-Z}$ is ample (resp. nef, big, pseudo-effective) as a vector bundle.

Next we introduce the reflexive sheaf and list here a few basic properties from \cite{Kob87}.
\begin{definition}\label{d24}
1. A coherent sheaf $\mathcal{E}$ is called reflexive if $\mathcal{E}\simeq\mathcal{E}^{\ast\ast}$.

2. A coherent sheaf $\mathcal{E}$ is called normal if for every open set $U$ in $X$ and every analytic subvariety $A\subset U$ of codimension at least $2$, the restriction map
\[
\Gamma(U,\mathcal{E})\rightarrow\Gamma(U-A,\mathcal{E})
\]
is isomorphism.
\end{definition}

\begin{proposition}\label{p21}
\begin{enumerate}
  \item[(a)] If $\mathcal{E}$ is reflexive, it is torsion free.
  \item[(b)] If $\mathcal{E}$ is reflexive, it is locally free outside a 3-codimensional subvariety.
  \item[(c)] If $\mathcal{E}$ is reflexive, it can be included in an exact sequence
  \begin{equation}\label{e21}
  0\rightarrow\mathcal{E}\rightarrow\mathcal{M}\rightarrow\mathcal{F}\rightarrow0
  \end{equation}
  with $\mathcal{M}$ locally free and $\mathcal{F}$ torsion free. Conversely, if $\mathcal{E}$ is included in an exact sequence (\ref{e21}) with $\mathcal{M}$ reflexive and $\mathcal{F}$ torsion free, then $\mathcal{E}$ is reflexive.
  \item[(d)] A coherent sheaf $\mathcal{E}$ is reflexive if and only if it is torsion free and normal.
\end{enumerate}
\end{proposition}

\subsection{The fibre product}
We introduce the fibre product here for the later use.
\begin{definition}\label{d25}
Let $f:X\rightarrow Y$ be a fibration between two projective manifolds $X$ and $Y$, the fibre product, denoted by $(X\times_{Y}X,p^{2}_{1},p^{2}_{2})$, is a projective manifold coupled with two morphisms (we will also call the manifold $X\times_{Y}X$ itself the fibre product if nothing is confused), which satisfies the following properties:

1.The diagram
\[
\begin{array}[c]{ccc}
X\times_{Y}X&\stackrel{p^{2}_{2}}{\rightarrow}&X\\
\scriptstyle{p^{2}_{1}}\downarrow&&\downarrow\scriptstyle{f}\\
X&\stackrel{f}{\rightarrow}&Y
\end{array}
\]
commutes.

2.If there is another projective manifold $Z$ with morphisms $q_{1}, q_{2}$ such that the diagram
\[
\begin{array}[c]{ccc}
Z&\stackrel{q_{2}}{\rightarrow}&X\\
\scriptstyle{q_{1}}\downarrow&&\downarrow\scriptstyle{f}\\
X&\stackrel{f}{\rightarrow}&Y
\end{array}
\]
commutes, then there must exist a unique $g:Z\rightarrow X\times_{Y}X$ such that $p^{2}_{1}\circ g=q_{1},p^{2}_{2}\circ g=q_{2}$.

We inductively define the $m$-fold fibre product, and denote it by $X\times_{Y}\cdot\cdot\cdot\times_{Y}X$. Moreover, we denote two projections by
\[
 p^{m}_{1}:X\times_{Y}\cdot\cdot\cdot\times_{Y}X\rightarrow X
\]
and
\[
 p^{m}_{2}:\underbrace{X\times_{Y}\cdot\cdot\cdot\times_{Y}X}_{m}\rightarrow \underbrace{X\times_{Y}\cdot\cdot\cdot\times_{Y}X}_{m-1}
\]
respectively.
\end{definition}

\section{The modified ideal sheaf}
\subsection{The definition}
We recall the modified ideal sheaf in \cite{Nad90} first. Remember that for a singular metric $\varphi$ on $L$ with analytic singularities, its multiplier ideal sheaf can be computed precisely. Indeed, suppose that
\[
\varphi\sim a\log(|f_{1}|^{2}+\cdots+|f_{N}|^{2})
\]
near the poles. Here $f_{i}$ is a holomorphic function. We define $\mathscr{S}$ to be the sheaf of
holomorphic functions $h$ such that $|h|^{2}e^{-\frac{\varphi}{a}}\leqslant C$. Then one computes a smooth modification $\mu:\tilde{X}\rightarrow X$ of $X$ such that $\mu^{\ast}\mathscr{S}$ is an invertible sheaf $\mathcal{O}_{\tilde{X}}(-D)$ associated with a normal crossing divisor $D=\sum\lambda_{j}D_{j}$, where $D_{j}$ is the component of the exceptional divisor of $\tilde{X}$. Now, we have $K_{\tilde{X}}=\mu^{\ast}K_{X}+R$, where $R=\sum\rho_{j}D_{j}$ is the zero divisor of the Jacobian function of the blow-up map. After some simple computation shown in \cite{Dem12}, we will finally get that
\[
\mathscr{I}(\varphi)=\mu_{\ast}\mathcal{O}_{\tilde{X}}(\sum (\rho_{j}-\lfloor a\lambda_{j}\rfloor)D_{j}),
\]
where $\lfloor a\lambda_{j}\rfloor$ denotes the round down of the real number $a\lambda_{j}$.

Now we have the following definition.
\begin{definition}\label{d31}
Let $h$ be a singular metric on $L$ with weight function $\varphi$. Assume that $\varphi$ has analytic singularities. Fix the notations as before,  the modified ideal sheaf as
\[
\mathfrak{I}(\varphi):=\mu_{\ast}\mathcal{O}_{\tilde{X}}(\sum(-\lceil a\lambda_{j}\rceil)D_{j}).
\]
Here $\lceil a\lambda_{j}\rceil$ denotes the round up of the real number $a\lambda_{j}$.
\end{definition}

It is not hard to see that $\mathfrak{I}(\varphi)$ is an ideal sheaf. Indeed, let
\[
\tau_{j}=
\begin{cases}
\lambda_{j}+\frac{\rho_{j}}{a} & \textrm{if } a\lambda_{j} \textrm{ is an integer} \\
\lambda_{j}+\frac{\rho_{j}+1}{a} & \textrm{if } a\lambda_{j} \textrm{ is not an integer}.
\end{cases}
\]
Then
\[
\mu_{\ast}\mathcal{O}_{\tilde{X}}(\sum(-\lceil a\lambda_{j}\rceil)D_{j})=\mu_{\ast}\mathcal{O}_{\tilde{X}}(\sum(\rho_{j}-\lfloor a\tau_{j}\rfloor)D_{j}).
\]
Let $(g_{j})$ be the local generators of $(D_{j})$ on $V_{i}$, we define
\[
\psi_{j}:=\mu_{\ast}(a\log\tau_{j}\sum|g_{j}|^{2}),
\]
which is a plurisubharmonic function on $U_{i}:=\mu(V_{i})$. Let $g^{j}_{ik}$ be the transition function of $\mathcal{O}_{\tilde{X}}(D_{j})$ between $V_{i}$ and $V_{k}$, then
\[
\psi_{i}=\psi_{k}+\mu_{\ast}(a\log\tau_{j}\sum|g^{j}_{ik}|^{2}).
\]
Since $g^{j}_{ik}$ is a nowhere vanishing holomorphic function, for any holomorphic function $f$ on $U_{i}\cap U_{kj}$ ,
\[
\int_{U_{i}\cap U_{k}}|f|^{2}e^{-\psi_{i}}<\infty
\]
iff
\[
\int_{U_{i}\cap U_{k}}|f|^{2}e^{-\psi_{k}}<\infty.
\]
Now we use a partition of unity $\{\theta_{i}\}$ to patch $\psi_{i}$ together to become a global function $\psi$ on $X$. It is then easy to verify that
\[
\mathfrak{I}(\varphi)=\mathscr{I}(\psi).
\]
 As a result, $\mathfrak{I}(\varphi)$ is an ideal sheaf.
 
 Next, using Demailly's approximation technique, it is easy to extend this definition to a general singular metric. Let $\varphi$ be a singular metric on $L$ not necessary with analytic singularities. Assume $i\Theta_{L,\varphi}\geqslant0$. Then by Demailly's approximation \cite{DPS01}, we can find a family of metrics $\{\varphi_{\varepsilon}\}$ on $L$ with the following properties:

(a) $\varphi_{\varepsilon}$ is smooth on $X-Z_{\varepsilon}$ for a subvariety $Z_{\varepsilon}$;

(b) $\varphi\leqslant\varphi_{\varepsilon_{1}}\leqslant\varphi_{\varepsilon_{2}}$ holds for any $0<\varepsilon_{1}\leqslant\varepsilon_{2}$;

(c) $\mathscr{I}(\varphi)=\mathscr{I}(\varphi_{\varepsilon})$; and

(d) $i\Theta_{L,\varphi_{\varepsilon}}\geqslant-\varepsilon\omega$.

Thanks to the proof of the openness conjecture by Berndtsson \cite{Ber15}, one can arrange $\phi_{\varepsilon}$ with logarithmic poles along $Z_{\varepsilon}$ according to the remark in \cite{DPS01}. Therefore, similar with the computation before Definition \ref{d22}, $\mathscr{I}(\varphi_{\varepsilon})$ can be rewritten as
\[
\mathscr{I}(\varphi_{\varepsilon})=\mu^{\varepsilon}_{\ast}\mathcal{O}_{\tilde{X}^{\varepsilon}}(\sum(\rho^{\varepsilon}_{j}-\lfloor a^{\varepsilon}\lambda^{\varepsilon}_{j}\rfloor)D^{\varepsilon}_{j}).
\]
Then we define
\[
\mathfrak{I}(\varphi)_{\varepsilon}:=\mu^{\varepsilon}_{\ast}\mathcal{O}_{\tilde{X}^{\varepsilon}}(\sum-\lceil a^{\varepsilon}\lambda^{\varepsilon}_{j}\rceil)D^{\varepsilon}_{j}).
\]
Hence the modified ideal sheaf for a singular metric $\varphi$ (not necessary with analytic singularities) is defined as
\begin{definition}\label{d32}
\[
\mathfrak{I}(\varphi):=\cap_{\varepsilon}\mathfrak{I}(\varphi)_{\varepsilon}.
\]
\end{definition}

Now we can prove Theorem \ref{t1}.
\begin{proof}[Proof of Theorem \ref{t1}]
(1) The first assertion is trivial. 

When $\varphi$ has analytic singularity, the second assertion has been shown above. For a general $\varphi$, we approximate it by $\varphi_{\varepsilon}$ as above. Then $\varphi_{\varepsilon}$ induces a function $\psi_{\varepsilon}$ with
\[
\mathfrak{I}(\varphi)_{\varepsilon}=\mathscr{I}(\psi_{\varepsilon}).
\]  
Let $\psi$ be the $L^{1}$-limit of $\psi_{\varepsilon}$. Obviously, 
\[
\mathfrak{I}(\varphi)=\mathscr{I}(\psi).
\]

Notice that $\mathscr{I}(\psi)=\mathscr{I}(\varphi)$ iff $\lambda^{\varepsilon}_{j}=\tau^{\varepsilon}_{j}$ for all $j$. In this situation, we have $\rho^{\varepsilon}_{j}=0$ and $a^{\varepsilon}\lambda^{\varepsilon}_{j}$ is an integer for all $j$, hence $a^{\varepsilon}$ is a rational number. Since $\rho^{\varepsilon}_{j}$ comes from the Jacobian divisor, it equals to zero iff the modification is trivial. In other word, the pole-set of $\varphi_{\varepsilon}$ is a normal crossing divisor $E$. On the other hand, $\varphi_{\varepsilon}$ is a metric of line bundle $L$, so it must be the case that $L=\mathcal{O}_{X}(E)$ and $\varphi$ is defined by $E$.

(2) Since $\varphi_{\varepsilon}\circ\mu^{\varepsilon}=a^{\varepsilon}\sum\lambda^{\varepsilon}_{j}\log|g^{\varepsilon}_{j}|^{2}$, for any $f\in\mathfrak{I}(\varphi)$ we have
\[
|f\circ\mu|^{2}e^{-\varphi_{\varepsilon}\circ\mu^{\varepsilon}}=|g|^{2}\Pi_{j}|g^{\varepsilon}_{j}|^{2(\lceil a^{\varepsilon}\lambda^{\varepsilon}_{j}\rceil-a^{\varepsilon}\lambda^{\varepsilon}_{j})}
\]
for a non-vanishing holomorphic function $g$. In particular, $|f|^{2}e^{-\varphi}$ do not vanish near the poles only if $\lceil a^{\varepsilon}\lambda^{\varepsilon}_{j}\rceil=a^{\varepsilon}\lambda^{\varepsilon}_{j}$ for all $\varepsilon$. In this case, $a^{\varepsilon}$ is a rational number, which proves the last assertion.

(3) is a direct consequence of (2).

In the proof of (4) and (5), we assume that the metrics have analytic singularities without loss of generality. 

(4) Suppose that
\[
\begin{split}
\psi_{1}&\sim a\log(|f_{1}|^{2}+\cdots+|f_{N}|^{2}),
\end{split}
\]
and
\[
\begin{split}
\psi_{2}&\sim b\log(|g_{1}|^{2}+\cdots+|g_{M}|^{2}),
\end{split}
\]
where $f_{i},g_{j}$ are holomorphic functions. We define $\mathscr{S},\mathscr{W}$ to be the sheaves of holomorphic functions $h,k$ such that
\[
|h|^{2}e^{-\frac{\psi_{1}}{a}}\leqslant C\textrm{ and }|k|^{2}e^{-\frac{\psi_{2}}{b}}\leqslant C
\]
respectively. Let $\mu_{1}:\tilde{X}_{1}\rightarrow X_{1}$ and $\mu_{2}:\tilde{X}_{2}\rightarrow X_{2}$ be the log-resolutions such that $\mu^{\ast}_{1}\mathscr{S},\mu^{\ast}_{2}\mathscr{W}$ are invertible sheaves
\[
\mathcal{O}_{\tilde{X}_{1}}(-D)\textrm{ and }\mathcal{O}_{\tilde{X}_{2}}(-F)
\]
associated with normal crossing divisors $D=\sum\lambda_{i}D_{i},F=\sum\tau_{i}F_{i}$. Here $D_{i},F_{j}$ are the components of the exceptional divisor of $\tilde{X}_{1},\tilde{X}_{2}$ respectively. Now, we have $K_{\tilde{X}_{k}}=\mu^{\ast}K_{X_{k}}+R_{k}$, with $k=1,2$, where $R_{1}=\sum\rho_{i}D_{i},R_{2}=\sum\sigma_{i}F_{i}$ are the corresponding zero divisors of the Jacobian function of the blow-up map. In this situation,
\[
\mathfrak{I}(\psi_{1})=(\mu_{1})_{\ast}\mathcal{O}_{\tilde{X}_{1}}(\sum(-\lceil a\lambda_{j}\rceil)D_{j})
\]
and
\[
\mathfrak{I}(\psi_{2})=(\mu_{2})_{\ast}\mathcal{O}_{\tilde{X}_{2}}(\sum(-\lceil b\tau_{j}\rceil)F_{j}).
\]

It is not hard to see that $\mu:\tilde{X}_{1}\times\tilde{X}_{2}\rightarrow X_{1}\times X_{2}$ is a log-resolution, and
\[
\begin{split}
&\mathfrak{I}(\psi_{1}\circ\pi_{1}+\psi_{2}\circ\pi_{2})\\
=&(\mu)_{\ast}\mathcal{O}_{\tilde{X}_{1}\times\tilde{X}_{2}}(\sum(-\lceil a\lambda_{j}\rceil)p^{\ast}_{1}D_{j}+\sum(-\lceil b\tau_{j}\rceil)p^{\ast}_{2}F_{j}).
\end{split}
\]
Here we use the fact that $K_{\tilde{X}_{1}\times\tilde{X}_{2}}=p^{\ast}_{1}K_{\tilde{X}_{1}}\otimes p^{\ast}_{2}K_{\tilde{X}_{2}}$, where
\[
p_{i}:\tilde{X}_{1}\times\tilde{X}_{2}\rightarrow\tilde{X}_{i}
\]
is the projection with $i=1,2$.

Thus we get that
\[
\begin{split}
&\pi^{\ast}_{1}\mathfrak{I}(\psi_{1})\cdot\pi_{2}^{\ast}\mathfrak{I}(\psi_{2})\\
=&\mu_{\ast}(p^{\ast}_{1}\mathcal{O}_{\tilde{X}_{1}}(\sum(-\lceil a\lambda_{j}\rceil)D_{j}))\otimes\mu_{\ast}(p^{\ast}_{2}\mathcal{O}_{\tilde{X}_{2}}(\sum(-\lceil b\tau_{j}\rceil)F_{j}))\\
=&(\mu)_{\ast}\mathcal{O}_{\tilde{X}_{1}\times\tilde{X}_{2}}(\sum(-\lceil a\lambda_{j}\rceil)p^{\ast}_{1}D_{j}+\sum(-\lceil b\tau_{j}\rceil)p^{\ast}_{2}F_{j})\\
=&\mathfrak{I}(\psi_{1}\circ\pi_{1}+\psi_{2}\circ\pi_{2}).
\end{split}
\]
The first equality is due to the base change theorem \cite{Har77}, and the second equality comes from the fact that $p_{1},p_{2}$ are smooth.

(5) The proof has been given in \cite{Wu20}. We present here for readers' convenience. Suppose that
\[
\begin{split}
\varphi_{1}&\sim a\log(|f_{1}|^{2}+\cdots+|f_{N}|^{2}),
\end{split}
\]
and
\[
\begin{split}
\varphi_{2}&\sim b\log(|g_{1}|^{2}+\cdots+|g_{M}|^{2}),
\end{split}
\]
where $f_{i},g_{j}$ are holomorphic functions. We define $\mathscr{S},\mathscr{W}$ to be the sheaves of holomorphic functions $h,k$ such that
\[
|h|^{2}e^{-\frac{\varphi_{1}}{a}}\leqslant C\textrm{ and }|k|^{2}e^{-\frac{\varphi_{2}}{b}}\leqslant C
\]
respectively. Let $\mu:\tilde{X}\rightarrow X$ be a log-resolution such that $\mu^{\ast}\mathscr{S},\mu^{\ast}\mathscr{W}$ are invertible sheaves $\mathcal{O}_{\tilde{X}}(-D_{1}),\mathcal{O}_{\tilde{X}}(-D_{2})$ associated with normal crossing divisors $D_{1}=\sum\lambda_{i}D_{i},D_{2}=\sum\tau_{i}D_{i}$. Here $D_{j}$ is the component of the exceptional divisor of $\tilde{X}$. Now, we have $K_{\tilde{X}}=\mu^{\ast}K_{X}+R$, where $R=\sum\rho_{i}D_{i}$ is the zero divisor of the Jacobian function of the blow-up map. Thus we get that
\[
\begin{split}
\mathfrak{I}(\varphi_{1})&=\mu_{\ast}\mathcal{O}_{\tilde{X}}(\sum(-\lceil a\lambda_{i}\rceil)D_{i}),\\
\mathfrak{I}(\varphi_{2})&=\mu_{\ast}\mathcal{O}_{\tilde{X}}(\sum(-\lceil b\tau_{i}\rceil)D_{i}),
\end{split}
\]
and
\[
\mathfrak{I}(\varphi_{1}+\varphi_{2})=\mu_{\ast}\mathcal{O}_{\tilde{X}}(\sum(-\lceil a\lambda_{i}+b\tau_{i}\rceil)D_{i}).
\]
Then the conclusion follows easily from the fact that
\[
\lceil a\lambda_{i}\rceil+\lceil b\tau_{i}\rceil\geqslant\lceil a\lambda_{i}+b\tau_{i}\rceil.
\]
\end{proof}

\subsection{The harmonic theory}
We develop the harmonic theory for the modified ideal sheaf associated with a pseudo-effective line bundle $L$.

Assume that $(X,\omega)$ is a compact K\"{a}hler manifold. Let $(L,\varphi)$ be a pseudo-effective line bundle on $X$. Assume that there exits a holomorphic section $s$ of $L^{k_{0}}$ for some integer $k_{0}$, such that $\sup_{X}|s|_{k_{0}\varphi}<\infty$. Then by Demailly's approximation technique \cite{DPS01}, there exits a family of metrics $\{\varphi_{\varepsilon}\}$ on $L$ with the following properties:

(a) $\varphi_{\varepsilon}$ is smooth on $X-Z_{\varepsilon}$ for a subvariety $Z_{\varepsilon}$;

(b) $\varphi\leqslant\varphi_{\varepsilon_{1}}\leqslant\phi_{\varepsilon_{2}}$ holds for any $0<\varepsilon_{1}\leqslant\varepsilon_{2}$;

(c) $\mathscr{I}(\varphi)=\mathscr{I}(\varphi_{\varepsilon})$; and

(d) $i\Theta_{L,\varphi_{\varepsilon}}\geqslant-\varepsilon\omega$.

Thanks to the proof of the openness conjecture by Berndtsson \cite{Ber15}, one can arrange $\varphi_{\varepsilon}$ with logarithmic poles along $Z_{\varepsilon}$ according to the remark in \cite{DPS01}. Moreover, since the norm $|s|_{k_{0}\varphi}$ is bounded on $X$, the set $\{x\in X|\nu(\varphi_{\varepsilon},x)>0\}$ for every $\varepsilon>0$ is contained in the subvariety $Z:=\{x|s(x)=0\}$ by property (b). Here $\nu(\varphi_{\varepsilon},x)$ refers to the Lelong number of $\varphi_{\varepsilon}$ at $x$. Hence, instead of (a), we can assume that

(a') $\varphi_{\varepsilon}$ is smooth on $X-Z$, where $Z$ is a subvariety of $X$ independent of $\varepsilon$.

Now let $Y=X-Z$. We use the method in \cite{Dem82} to construct a complete K\"{a}hler metric on $Y$ as follows. Since $Y$ is weakly pseudo-convex, we can take a smooth plurisubharmonic exhaustion function $\psi$ on $X$. Define $\tilde{\omega}=\omega+\frac{1}{l}i\partial\bar{\partial}\psi^{2}$ for $l\gg0$. It is easy to verify that $\tilde{\omega}$ is a complete K\"{a}hler metric on $Y$ and $\tilde{\omega}\geqslant\omega$.

Let $L^{n,q}_{(2)}(Y,L)_{\varphi_{\varepsilon},\tilde{\omega}}$ be the $L^{2}$-space of $L$-valued $(n,q)$-forms $\alpha$ on $Y$ with respect to the inner product given by $\varphi_{\varepsilon},\tilde{\omega}$. Then we have the orthogonal decomposition
\begin{equation}\label{e31}
L^{n,q}_{(2)}(Y,L)_{\varphi_{\varepsilon},\tilde{\omega}}=\mathrm{Im}\bar{\partial}\bigoplus\mathcal{H}^{n,q}_{\varphi_{\varepsilon}, \tilde{\omega}}(L)\bigoplus\mathrm{Im}\bar{\partial}^{\ast}_{\varphi_{\varepsilon}}
\end{equation}
where
\[
  \mathcal{H}^{n,q}_{\varphi_{\varepsilon}, \tilde{\omega}}(L)=\{\alpha|\bar{\partial}\alpha=0, \bar{\partial}^{\ast}_{\varphi_{\varepsilon}}\alpha=0\}.
\]
We give a brief explanation for decomposition (\ref{e31}). Usually $\mathrm{Im}\bar{\partial}$ is not closed in the $L^{2}$-space of a noncompact manifold even if the metric is complete. However, in the situation we consider here, $Y$ has the compactification $X$, and the forms on $Y$ are bounded in $L^{2}$-norms. Such a form will have good extension properties. Therefore the set $L^{n,q}_{(2)}(Y,L)_{\varphi_{\varepsilon},\tilde{\omega}}\cap\mathrm{Im}\bar{\partial}$ behaves much like the space $\mathrm{Im}\bar{\partial}$ on $X$, which is surely closed. The complete explanation can be found in \cite{Fuj12,Wu17}.

Now we have all the ingredients for the definition of $\Box_{\varphi}$-harmonic forms. We denote the Lapalcian operator on $Y$ associated to $\tilde{\omega}$ and $\varphi_{\varepsilon}$ by $\Box_{\varepsilon}$.
\begin{definition}\label{d33}
Let $\alpha$ be a $\bar{\partial}$-closed $L$-valued $(n,q)$-form on $X$ with bounded $L^{2}$-norm with respect to $\omega,\varphi$. Assume that for every $\varepsilon\ll1$, there exists a representative $\alpha_{\varepsilon}\in[\alpha|_{Y}]$ such that
\begin{enumerate}
  \item $\Box_{\varepsilon}\alpha_{\varepsilon}=0$ on $Y$;
  \item $\alpha_{\varepsilon}\rightarrow\alpha|_{Y}$ in $L^{2}$-norm.
\end{enumerate}
Then we call $\alpha$ a $\Box_{\varphi}$-harmonic form. The space of all the $\Box_{\varphi}$-harmonic forms is denoted by
\[
\mathcal{H}^{n,q}(X,L\otimes\mathscr{I}(\varphi),\Box_{\varphi}).
\]
\end{definition}

We will show that Definition \ref{d32} is compatible with the usual definition of $\Box_{\varphi}$-harmonic forms for a smooth $\varphi$ by proving the Hodge-type isomorphism, i.e. Theorem \ref{t2}.

\begin{proof}[Proof of Theorem \ref{t2}]
We use the de Rham--Weil isomorphism
\[
H^{n,q}(X,L\otimes\mathscr{I}(\varphi))\cong\frac{\mathrm{Ker}\bar\partial\cap L^{n,q}_{(2)}(X,L)_{h,\omega}}{\mathrm{Im}\bar{\partial}}
\]
to represent a given cohomology class $[\alpha]\in H^{n,q}(X,L\otimes\mathscr{I}(\varphi))$ by a $\bar{\partial}$-closed $L$-valued $(n,q)$-form $\alpha$ with $\|\alpha\|_{\varphi,\omega}<\infty$. We denote $\alpha|_{Y}$ simply by $\alpha_{Y}$. Since $\tilde{\omega}\geqslant\omega$, it is easy to verify that
\[
|\alpha_{Y}|^{2}_{\varphi_{\varepsilon},\tilde{\omega}}dV_{\tilde{\omega}}\leqslant|\alpha|^{2}_{\varphi_{\varepsilon},\omega}dV_{\omega},
\]
which leads to inequality $\|\alpha_{Y}\|_{\varphi_{\varepsilon},\tilde{\omega}}\leqslant\|\alpha\|_{\varphi_{\varepsilon,\omega}}$ with $L^{2}$-norms. Hence by property (b), we have
$\|\alpha_{Y}\|_{\varphi_{\varepsilon},\tilde{\omega}}\leqslant\|\alpha\|_{\varphi,\omega}$ which implies
\[
\alpha_{Y}\in L^{n,q}_{(2)}(Y,L)_{\varphi_{\varepsilon},\tilde{\omega}}.
\]
By decomposition (\ref{e31}), we have a harmonic representative $\alpha_{\varepsilon}$ in
\[
\mathcal{H}^{n,q}_{\varphi_{\varepsilon,\tilde{\omega}}}(L),
\]
which means that $\Box_{\varepsilon}\alpha_{\varepsilon}=0$ on $Y$ for all $\varepsilon$. Moreover, since a harmonic representative minimizes the $L^{2}$-norm, we have
\[
   \|\alpha_{\varepsilon}\|_{\varphi_{\varepsilon},\tilde{\omega}}\leqslant\|\alpha_{Y}\|_{\varphi_{\varepsilon},\tilde{\omega}}\leqslant \|\alpha\|_{\varphi,\omega}.
\]
So there exists a limit $\tilde{\alpha}$ of (a subsequence of) $\{\alpha_{\varepsilon}\}$ such that
\[
\tilde{\alpha}\in[\alpha_{Y}].
\]
It is left to extend it to $X$.

Indeed, by (the proof of) Proposition 2.1 in \cite{Wu17}, there is an injective morphism, which maps $\tilde{\alpha}$ to a $\bar{\partial}$-closed $L$-valued $(n-q,0)$-form on $Y$ with bounded $L^{2}$-norm. We formally denote it by $\ast\tilde{\alpha}$. The canonical extension theorem applies here and $\ast\tilde{\alpha}$ extends to a $\bar{\partial}$-closed $L$-valued $(n-q,0)$-form on $X$, which is denoted by $S^{q}(\tilde{\alpha})$ in \cite{Wu17}. Furthermore, it is shown by Proposition 2.2 in \cite{Wu17} that $\hat{\alpha}:=c_{n-q}\omega_{q}\wedge S^{q}(\tilde{\alpha})$ is an $L$-valued $(n,q)$-form with
\[
  \hat{\alpha}|_{Y}=\tilde{\alpha}.
\]
Therefore we finally get an extension $\hat{\alpha}$ of $\tilde{\alpha}$. By definition,
\[
\hat{\alpha}\in\mathcal{H}^{n,q}(X,L\otimes\mathscr{I}(\varphi),\Box_{\varphi}).
\]
We denote this morphism by $i([\alpha])=\hat{\alpha}$.

On the other hand, for a given $\alpha\in\mathcal{H}^{n,q}(X,L\otimes\mathscr{I}(\phi),\Box_{\varphi})$, by definition there exists an $\alpha_{\varepsilon}\in[\alpha_{Y}]$ with $\alpha_{\varepsilon}\in\mathcal{H}^{n,q}_{\varphi_{\varepsilon,\tilde{\omega}}}(L)$ and $\lim\alpha_{\varepsilon}=\alpha_{Y}$ for every $\varepsilon$. In particular, $\bar{\partial}\alpha_{\varepsilon}=0$. So all of the $\alpha_{\varepsilon}$ together with $\alpha_{Y}$ define a common cohomology class $[\alpha_{Y}]$ in $H^{n,q}(Y,L\otimes\mathscr{I}(\varphi))$. It is left to extend this class to $X$.

We use the $S^{q}$ again. It maps $[\alpha_{Y}]$ to
\[
S^{q}(\alpha_{Y})\in H^{0}(X,\Omega^{n-q}_{X}\otimes L\otimes\mathscr{I}(\varphi)).
\]
Then
\[
   c_{n-q}\omega_{q}\wedge S^{q}(\alpha_{Y})\in H^{n,q}(X,L\otimes\mathscr{I}(\varphi))
\]
with $[(c_{n-q}\omega_{q}\wedge S^{q}(\alpha_{Y}))|_{Y}]=[\alpha_{Y}]$ as a cohomology class in
\[
H^{n,q}(Y,L\otimes\mathscr{I}(\varphi)).
\]
Here we use the fact that $\omega$ is a K\"{a}hler metric. We denote this morphism by $j(\alpha)=[c_{n-q}\omega_{q}\wedge S^{q}(\alpha_{Y})]$. It is easy to verify that $i\circ j=\textrm{id}$ and $j\circ i=\textrm{id}$. The proof of first isomorphism is finished.

Before proving the second inclusion, we need to introduce a lemma.
\begin{lemma}\label{l31}
If $\alpha\in H^{n,q}(X,L\otimes\mathfrak{I}(\varphi))$, then its $\Box_{0}$-harmonic representative $\tilde{\alpha}$ satisfies that
\[
|\tilde{\alpha}|^{2}_{h_{0}}\leqslant|\alpha|^{2}_{h_{0}}\textrm{ near the poles}.
\]
\end{lemma}
The proof can be found in \cite{Wu20}. Now let
\[
\textrm{Im}\bar{\partial}_{1}:=\textrm{Im}(\bar{\partial}:L^{n,q}_{(2)}(X,L)_{\varphi}\rightarrow L^{n,q}_{(2)}(X,L)_{\varphi}),
\]
and
\[
\textrm{Im}\bar{\partial}_{2}:=\textrm{Im}(\bar{\partial}:L^{n,q}_{(2)}(X,L)_{\psi}\rightarrow L^{n,q}_{(2)}(X,L)_{\psi}).
\]
$\textrm{Ker}\bar{\partial}_{1}$ and $\textrm{Ker}\bar{\partial}_{2}$ are defined similarly. Recall that there are following orthogonal decompositions \cite{GHS98}:
\[
\textrm{Ker}\bar{\partial}_{1}=\textrm{Im}\bar{\partial}_{1}\bigoplus\mathcal{H}^{n,q}(X,L\otimes\mathscr{I}(\varphi),\Box_{0})
\]
and
\[
\textrm{Ker}\bar{\partial}_{2}=\textrm{Im}\bar{\partial}_{2}\bigoplus(\textrm{Ker}\bar{\partial}_{2}\cap\textrm{Ker}\bar{\partial}^{\ast}_{\psi}).
\]
By $\bar{\partial}^{\ast}_{\psi}$ we refer to the formal adjoint operator of $\bar{\partial}$ with respect to the $L^{2}$-norm defined by $\varphi$. One may wonder that are these decompositions still valid for singular metrics. Indeed, we can approximate them by smooth metrics then take the limit. On the other hand, it is easy to verify that
\begin{equation}\label{e32}
\textrm{Ker}\bar{\partial}_{1}\cap L^{n,q}_{(2)}(X,L)_{\psi}=\textrm{Ker}\bar{\partial}_{2}.
\end{equation}

Now the cohomology group can be expressed as
\[
H^{n,q}(X,L\otimes\mathscr{I}(\varphi))\simeq\frac{\textrm{Ker}\bar{\partial}_{1}}{\textrm{Im}\bar{\partial}_{1}}= \mathcal{H}^{n,q}(X,L\otimes\mathscr{I}(\varphi),\Box_{0}))
\]
and
\[
H^{n,q}(X,L\otimes\mathfrak{I}(\varphi))\simeq\frac{\textrm{Ker}\bar{\partial}_{2}}{\textrm{Im}\bar{\partial}_{2}}= \textrm{Ker}\bar{\partial}_{2}\cap\textrm{Ker}\bar{\partial}^{\ast}_{\psi}.
\]
Therefore the morphism $i_{n,q}$ can be rewritten as
\[
i_{n,q}:\textrm{Ker}\bar{\partial}_{2}\cap\textrm{Ker}\bar{\partial}^{\ast}_{\psi}\rightarrow\mathcal{H}^{n,q}(X,L\otimes \mathscr{I}(\varphi),\Box_{0})),
\]
hence its image equals
\[
\begin{split}
&\frac{\textrm{Ker}\bar{\partial}_{2}\cap\textrm{Ker}\bar{\partial}^{\ast}_{\psi}}{\textrm{Im}\bar{\partial}_{1}}\\
=&\frac{\textrm{Ker}\bar{\partial}_{2}/\textrm{Im}\bar{\partial}_{2}}{\textrm{Im}\bar{\partial}_{1}}=\frac{\textrm{Ker}\bar{\partial}_{2}} {\textrm{Im}\bar{\partial}_{1}}\\
=&\frac{\textrm{Ker}\bar{\partial}_{1}\cap L^{n,q}_{(2)}(X,L)_{\psi}}{\textrm{Im}\bar{\partial}_{1}}\\
=&\mathcal{H}^{n,q}(X,L\otimes\mathscr{I}(\varphi),\Box_{0}))\cap L^{n,q}_{(2)}(X,L)_{\psi}.
\end{split}
\]
We use the fact that
\[
\textrm{Im}\bar{\partial}_{2}\subset\textrm{Im}\bar{\partial}_{1}
\]
to get the second equality. The third equality comes from formula (\ref{e32}) and the last equality is due to Lemma \ref{l31}. The proof of Theorem \ref{t2} is finished.
\end{proof}
\begin{remark}\label{r31}
It is easy to see that the proof for the first isomorphism also works when $\varphi$ doesn't have analytic singularities. Moreover, when $L$ is nef, we can even prove this singular Hodge's theorem in a general compact complex manifold. One refers for our paper \cite{Wu20} for more details.
\end{remark}

We are willing to know the following question.
\begin{problem}
Do we have
\[
\mathcal{H}^{n,q}(X,L\otimes\mathfrak{I}(\varphi),\Box_{\varphi})=H^{n,q}(X,L\otimes\mathfrak{I}(\varphi)).
\]
or not? If so, $i_{n,q}$ must be injective.
\end{problem}

\section{Application on the fibration}
This section is devoted to prove Theorem \ref{t3}. Let $f:X\rightarrow Y$ be a fibration between two compact K\"{a}hler manifolds, and let $(L,\varphi)$ be a pseudo-effective line bundle on $X$. First we prove Theorem \ref{t3}.
\begin{proof}[Proof of Theorem \ref{t3}]
1. We only need to prove that it's a well-defined morphism. Since $\alpha\in\mathcal{H}^{n,q}(X,L\otimes\mathfrak{I}(\varphi))$ and $s\in H^{0}(X,L^{k-1}\otimes\mathfrak{I}((k-1)\varphi))$,
\[
s\alpha\in\mathcal{H}^{n,q}(X,L^{k}\otimes\mathfrak{I}(k\varphi))
\]
by Theorem \ref{t2}.

2. Recall the map $S^{q}$ in the proof of Theorem \ref{t2}, which maps the element in
\[
\mathcal{H}^{n,q}(X,L\otimes\mathfrak{I}(\varphi))
\]
to $H^{0}(X,\Omega^{n-q}_{X}\otimes L)$. We define a similar map on $X_{y}$. That is
\[
S^{q}_{y}:\mathcal{H}^{n,q}(X_{y},L_{y}\otimes\mathfrak{I}(\varphi|_{X_{y}}))\rightarrow H^{0}(X_{y},\Omega^{n-q}_{X_{y}}\otimes L_{y}).
\]
It naturally lifts to a morphism on the higher direct images, which is denoted as
\[
S^{q}:\mathrm{Im}i_{n,q}\otimes L\otimes\mathfrak{I}(\varphi))\rightarrow f_{\ast}(\Omega^{n-q}_{X}\otimes L).
\]
Moreover, the map
\[
\begin{split}
L^{q}_{y}:H^{0}(X,\Omega^{n-q}_{X_{y}}\otimes L_{y})&\rightarrow H^{q}(X_{y},K_{X_{y}}\otimes L_{y})\\
\beta&\mapsto[\beta\wedge\omega^{q}_{y}]
\end{split}
\]
also lifts to a morphism $L^{q}$ on the direct images. It is easy to verify that $L^{q}\circ S^{q}=id$. Hence $S^{q}$ is split.

Set $\mathcal{Q}^{q}=\ker L^{q}$, we have short exact sequence
\[
0\rightarrow\mathrm{Im}i_{n,q}\rightarrow f_{\ast}(\Omega^{n-q}_{X}\otimes L)\rightarrow\mathcal{Q}^{q}\rightarrow0.
\]
By Corollary 1.7 in \cite{Har80}, $f_{\ast}(\Omega^{n-q}_{X}\otimes L)$ is reflexive and hence $Q^{q}$ is torsion free. Therefore
$\mathrm{Im}i_{n,q}$ is normal and so reflexive.
\end{proof}

\address{

\small Current address: School of Mathematical Sciences, Fudan University, Shanghai 200433, People's Republic of China.

\small E-mail address: jingcaowu08@gmail.com, jingcaowu13@fudan.edu.cn
}

\end{document}